\documentstyle[12pt]{article}
\newtheorem{theorem}{Theorem}[section]
\newtheorem{lemma}[theorem]{Lemma}
\newtheorem{proposition}[theorem]{Proposition}
\newtheorem{example}[theorem]{Example}
\newtheorem{corollary}[theorem]{Corollary}

\newtheorem{remar}[theorem]{Remark}
\newtheorem{prob}{Open Problem}[section]
\newenvironment{proof}{Proof:\ \ \ }{\QED}

\newcommand{\QED}{{\unskip\nobreak\hfil\penalty50%
\hskip1em\hbox{}\nobreak\hfil $\Box$%
\parfillskip=0pt \finalhyphendemerits=0 \par\medskip\noindent}}
\newcommand{\bfind}[1]{\index{#1}{\bf #1}}

\newcommand{\n}{\par\noindent}

\newcommand{\sn}{\par\smallskip\noindent}
\newcommand{\mn}{\par\medskip\noindent}
\newcommand{\bn}{\par\bigskip\noindent}
\newcommand{\pars}{\par\smallskip}
\newcommand{\parm}{\par\medskip}

\newcommand{\isom}{\simeq}

\newcommand{\ovl}[1]{\overline{#1}}

\newcommand{\chara}{\mbox{\rm char}\,}
\newcommand{\trdeg}{\mbox{\rm trdeg}\,}

\newcommand{\rr}{\mbox{\rm rr}\,}

\newcommand{\fvkadresse}{\par\bigskip \small\rm
 Department of Mathematics and Statistics, 
 University of Saskatchewan, \par
 106 Wiggins Road, 
 Saskatoon, Saskatchewan, Canada S7N 5E6 \par
 email: fvk@math.usask.ca}

\newcommand{\anadresse}{\par\bigskip \small\rm
 Department of Mathematics, 
 GC University, \par
 Katchery Road, 
 Lahore, Pakistan 54000 \par
 email: asimroz@gmail.com}

\font\tenlv=msbm10 scaled 1200
\font\sevenlv=msbm7 scaled 1200
\font\fivelv=msbm5 scaled 1200

\def\lv #1{{\mathchoice{{\hbox{\tenlv #1}}}{{\hbox{\tenlv #1}}}
{{\hbox{\sevenlv #1}}}{{\hbox{\fivelv #1}}}}}

\newcommand{\Q}{\lv Q}

\newcommand{\Z}{\lv Z}
\newcommand{\F}{\lv F}
\newcommand{\Fp}{\F_p}
\begin{document}
\title{Defects of Algebraic Function Fields, Completion Defects and
Defect Quotients}
\author{Franz-Viktor Kuhlmann, Asim Naseem}
\date{22.\ 1.\ 2013}
\maketitle
\begin{abstract}\noindent
The {\it defect} (also called {\it ramification deficiency}) of valued
field extensions is a major stumbling block in deep open problems of
valuation theory in positive characteristic. For a detailed analysis, we
define and investigate two weaker notions of defect: the {\it completion
defect} and the {\it defect quotient}. We define them for finite
extensions as well as for certain valued function fields (those with
``Abhyankar valuations'' that are nontrivial on the ground field). We
investigate both defects and present analogues of the results that hold
for the usual defect.
\end{abstract}
%
%
%
%
\section{Introduction}
For a valued field $(K,v)$ we denote its value group by $vK$ and its
residue field by $Kv$; for $a\in K$, $va$ denotes its value, and $av$
its residue. We denote the algebraic closure of $K$ by $\tilde{K}$ and
the perfect hull by $K^{1/p^{\infty}}$. By $(L|K,v)$ we mean an
extension of valued fields where $v$ is a valuation on $L$ and its
subfield $K$ is endowed with the restriction of $v$. Throughout,
\textbf{function field} will always mean an {\it algebraic function
field}.

In what follows, we fix an extension of the valuation $v$ from $K$ to
its algebraic closure. All algebraic extensions of $K$ will be endowed
with the restriction of this valuation. All of these valuations will
again be denoted by $v$. This also determines uniquely the
henselizations of all algebraic extensions of $K$ (cf.\
Section~\ref{secthc}). An algebraic extension $(L|K,v)$ is called
\textbf{h-finite} if $(L^h|K^h,v)$ is finite, where $K^h$ is the unique
henselization of $K$ inside the henselization $L^h$ of $L$. The defect
of an h-finite extension $(L|K,v)$ is the natural number
\[
d(L|K,v)\>:=\>\frac{[L^h:K^h]}{(vL^h:vK^h)[L^hv:K^hv]}
\>=\>\frac{[L^h:K^h]}{(vL:vK)[Lv:Kv]}\>;
\]
the second equation holds since henselizations are immediate extensions
(see Section~\ref{secthc}). By the Lemma of Ostrowski (see
Section~\ref{secDefect}), this quotient is always 1 if $Kv$ has
characteristic 0, and it is a power of $p$ (possibly 1) if $Kv$ has
characteristic $p>0$.

Matignon and Ohm (cf.\ [M], [O]) have used a \textsl{completion defect}
which measures the defect (for valuations of rank 1) using completions
instead of henselizations; it can be analogously defined for certain
valued function fields. The notion of completion defect played the key
role in the proof of Matignon's genus reduction inequality for valued
function fields. This inequality was first proved by Matignon in [M] for
valued function fields of transcendence degree 1 and rank 1 (i.e., with
archimedean value group, meaning that it is a subgroup of the reals). It
was extended in [GMP] to an arbitrary finite family of valuations
coinciding on the constant field.

Since for valuations of arbitrary rank, the completion
is in general not henselian, we measure the defect over the completion
of the henselization. This defect then coincides with Matignon's and
Ohm's completion defect if the valuation is of rank 1.

Let $K^{hc}$ denote the completion of $(K^h,v)$. Take any h-finite
extension $(L|K,v)$. We define the \textbf{completion defect}
$d_c(L|K,v)$ as follows:
\[
d_c(L|K,v)\>:=\>\frac{[L^{hc}:K^{hc}]}{(vL:vK)[Lv:Kv]}\>=\>
d(L^{hc}|K^{hc},v)\>;
\]
this is well-defined and the second equation holds because completions
are immediate extensions and $(L^{hc}|K^{hc},v)$ is a finite extension
of henselian fields (see Section~\ref{secthc}, in particular
Lemma~\ref{Lc=L.Kc} and Lemma~\ref{hchc=hc}). This is why we have chosen
to work with the completion of the henselization. Instead, we could have
chosen to work with
\[
d(L^c|K^c,v)\>=\>\frac{[L^{ch}:K^{ch}]}{(vL^c:vK^c)[L^cv:K^cv]}
\>=\>\frac{[L^ch:K^ch]}{(vL:vK)[Lv:Kv]},
\]
where $K^{ch}$ is the henselization of the completion of $K$ (which is
not necessarily complete). In fact, we will prove:

\begin{proposition}                             \label{ddcdch}
For every h-finite extension $(L|K,v)$,
\[
d(L^c|K^c,v)\>=\>d_c(L|K,v)\>,
\]
and for every h-finite separable extension $(L|K,v)$,
\[
d_c(L|K,v)\>=\>d(L|K,v)\>.
\]
\end{proposition}

In order to characterize those extensions for which the completion
defect is equal to the ordinary defect we compute the \textbf{defect
quotient}:
\[
d_q(L|K,v)\>:=\>\frac{d(L|K,v)}{d_c(L|K,v)}\>=\>
\frac{[L^h:K^h]}{[L^{hc}:K^{hc}]}\>.
\]
We denote by $[L : K]_{\rm insep}$ the {\bf inseparable degree} of
a finite extension $L|K$, that is, the degree of $L|K$ divided by the
degree of the maximal separable subextension $L_s|K$.

\begin{proposition}                         \label{cens}
For every finite extension $(L|K,v)$,
\[
d_q(L|K) = \frac{[L : K]_{\rm insep}}{[L^c : K^c]_{\rm insep}}\;.
\]
\end{proposition}

An h-finite extension $(L|K,v)$ is called \textbf{c-defectless} if
$d_c(L|K,v)=1$ and \textbf{q-defectless} if $d_q(L|K,v)=1$. A valued
field $(K,v)$ will be called a \textbf{c-defectless field} or
\textbf{q-defectless field} if every finite extension $(L|K,v)$ is
c-defectless or q-defectless, respectively. The properties of being
``c-defectless" or ``q-defectless" are weaker than ``defectless" (cf.
Section \ref{secDefect} for the latter notion). Another pair of weaker
properties are inseparably defectless and separably defectless (again,
see Section \ref{secDefect} for these notions).

\parm
In Section \ref{sechF}, we prove the following characterizations:

\begin{theorem}                                    \label{ims}
A valued field $(K,v)$ is a c-defectless field if and only if it is
a separably defectless field.
\end{theorem}

\begin{theorem}                                    \label{tcs}
A valued field $(K,v)$ is q-defectless if and only if its completion
is a separable extension (that is, linearly disjoint from the perfect
hull $K^{1/p^{\infty}}$ over $K$). In particular, every complete field
and every valued field of characteristic 0 is q-defectless.
\end{theorem}

\parm
A valued field extension $(F|K,v)$ of finite transcendence degree is
called \textbf{without transcendence defect} if equality holds in the
{\bf Abhyankar Inequality}
\begin{equation}                            \label{wtdgeq}
\trdeg F|K \>\geq\> \trdeg Fv|Kv \,+\, \rr vF/vK\;,
\end{equation}
where for any ordered abelian group $G$, $\rr \ G:=\dim_{\Q}G\otimes\Q$
denotes the maximal number of rationally independent elements in $G$;
this is called the \textbf{rational rank} of $G$. In particular in the
case where $v$ is trivial on $K$, valuations without transcendence
defect are also called {\bf Abhyankar valuations}.

Every extension $(F|K,v)$ without transcendence defect admits a
\textbf{standard valuation transcendence basis}, that is, a
transcendence basis $\{x_i,y_j \mid i\in I, j\in J\}$ such
that
\begin{equation}                           \label{valindep}
\left.\begin{array}{l}
\mbox{the values $vx_i\,$, $i\in I$, are rationally independent over
$vK$,}\\
\mbox{the residues $y_jv$, $j\in J$, are algebraically independent over
$Kv$.}
\end{array}\right\}
\end{equation}
(The second condition implicitly says that $vy_j=0$ for all $j\in J$.)
Indeed, if the elements $x_i$ are chosen such that their values form a
maximal set of elements in $vF$ rationally independent modulo $vK$, and
the elements $y_j$ are chosen such that their residues form a
transcendence basis of $Fv|Kv$, then by Lemma~\ref{K(T)}, the elements
$x_i,y_j$, $i\in I$, $j\in J$, are algebraically independent over $K$.
If equality holds in the Abhyankar Inequality, then their number equals
the transcendence degree, and so they form a standard valuation
transcendence basis.

\pars
As for h-finite extensions, a notion of defect can be defined for every
valued function field $(F|K,v)$ without transcendence defect. More
generally, we consider \textbf{subhenselian function fields}, that is,
extensions $(F|K,v)$ for which $(F,v)^h$ is the henselization of some
valued function field $(F_0|K,v)$. Note that in this case, $F|K(T)$
is an h-finite extension for every transcendence basis $T$ of $F|K$ and
hence we may consider the defect $d(F|K(T),v)$. For subhenselian
function fields without transcendence defect, the defect
can be introduced as:

\begin{equation}                                    \label{Tdef}
d(F|K,v) \>:=\> \sup_T\>d(F|K(T),v)
\end{equation}
where the supremum is taken over all transcendence bases of $(F|K,v)$.
The defect $d(F|K,v)$ is indeed a finite number whenever $(F|K,v)$ is
without transcendence defect, because it is equal to $d(F|K(T),v)$
whenever $T$ is a standard valuation transcendence basis. The following
theorem shows the \bfind{finiteness} of the defect and its
\bfind{independence} of the choice of the standard valuation
transcendence basis.

\begin{theorem}[Finiteness and Independence Theorem]     \label{def}
Take a subhenselian function field $(F|K,v)$ without transcendence
defect. Then for every standard valuation transcendence basis $T$ of
$F|K$,
\begin{equation}                                 \label{defe}
d(F|K,v) \>=\> d(F|K(T),v) < \infty \;.\label{-m}
\end{equation}
Moreover, there exists a finite extension $K'$ of $K$ such that for
every algebraic\ extension $L$ of $K$ containing $K'$ we have:
\begin{enumerate}
  \item for every standard valuation transcendence basis $T$ of
        $(F|K,v)$,\\ the extension $(L.F|L(T),v)$ is defectless,
  \item ${\displaystyle d(F|K,v) = \frac{d(L|K,v)}{d(L.F|F,v)} =
        \max_{N|K \mbox{\rm\scriptsize\ finite}}\ \
        \frac{d(N|K,v)}{d(N.F|F,v)}}\;$.
\end{enumerate}
\end{theorem}

The proof of this theorem is given in Section \ref{secDefFF}. It heavily
depends on the following result proved in [K1] and [K4]:

\begin{theorem}[Generalized Stability Theorem]     \label{thmGST}
Let $(F|K,v)$ be a valued function field without transcendence defect.
If $(K,v)$ is a defectless field, then also $(F,v)$ is a defectless
field. The same holds for ``inseparably defectless" in place of
``defectless". If $vK$ is cofinal in $vF$, then it also holds for
``separably defectless" in place of ``defectless".
\end{theorem}

This theorem is a generalization of a result of Grauert and Remmert
[G-R] which is restricted to the case of algebraically closed complete
ground fields of rank 1. A first generalization of the result was given
by Gruson [GRU]; an improved presentation of it can be found in the book
[BGR] of Bosch, G\"untzer and Remmert. Further generalizations are due
to M. Matignon and J. Ohm; see also [GMP]. In [O], Ohm arrived at a
version of Theorem~\ref{thmGST} with the restriction that $\trdeg
F|K=\trdeg Fv|Kv$. The work of all authors mentioned above is based on
methods of nonarchimedean analysis. In contrast, the proofs given in
[K1] and [K4] are purely valuation theoretic.

\pars
Theorem~\ref{def} was independently obtained by Ohm [O] in the case
where $\trdeg F|K=\trdeg Fv|Kv$ by using his version of the stability
theorem. He gave this result the name {\sl Independence Theorem} since
it shows the independence of the defect from the chosen standard
valuation transcendence basis.

Another special case was proved by Sudesh Khanduja in [Kh]. She
considered simple transcendental extensions $(K(x)|K,v)$
satisfying the condition $\rr vK(x)|vK=1=\trdeg K(x)|K$.

\parm
The completion defect and defect quotient of valued function fields
without transcendence defect may be defined similarly as it was done for
the defect. For a subhenselian function field $(F|K,v)$ without
transcendence defect we set
\[
d_c(F|K,v):=\sup_T d_c(F|K(T),v)\;\;\mbox{\ \ and\ \ }\;\;
d_q(F|K,v):=\sup_T d_q(F|K(T),v)\>,
\]
where the supremum is taken over all transcendence bases of $(F|K,v)$.
The same finiteness and independence as for the defect
(Theorem~\ref{def}) also hold for the completion defect and defect
quotient:

\begin{theorem}                                           \label{def'}
Take a subhenselian function field $(F|K,v)$ without transcendence
defect over $K$. Then for every standard valuation transcendence basis
$T$ of $F|K$,
\begin{equation}                               \label{m}
d_c(F|K,v) \>=\> d_c(F|K(T),v)\;\;\mbox{\ \ and\ \ }\;\;
d_q(F|K,v) \>=\> d_q(F|K(T),v)\;.
\end{equation}
Further,
\begin{equation}                                      \label{mu}
d(F|K,v) \>=\> d_c(F|K,v)\cdot d_q(F|K,v)\;.
\end{equation}
Assume in addition that $vK$ is cofinal in $vF$. If $K'$ is chosen as in
the assertion of Theorem~\ref{def}, then for every finite extension $L$
of $K$ containing $K'$,
\begin{eqnarray}
d_c(F|K,v) & = & \frac{d_c(L|K,v)}{d_c(L.F|F,v)}
\>=\> \max_{N|K \mbox{\rm\scriptsize\ finite}}
\frac{d_c(N|K,v)}{d_c(N.F|F,v)}                         \label{m1}\\
d_q(F|K,v) & = & \frac{d_q(L|K,v)}{d_q(L.F|F,v)}
\>=\> \max_{N|K \mbox{\rm\scriptsize\ finite}}
\frac{d_q(N|K,v)}{d_q(N.F|F,v)}\;.                      \label{m3}
\end{eqnarray}
\end{theorem}

We will prove Theorem~\ref{def'} in Section~\ref{secSHFF}, together with
the following ``q-defectless and c-defectless versions" of Theorem
\ref{thmGST}.

\begin{theorem}                             \label{313}
Take a subhenselian function field $(F|K,v)$ without transcendence
defect.
\sn
a) \ If $(K,v)$ is a q-defectless field or $vK$ is not cofinal in $vF$,
then $d_q(F|K,v)=1$ and $(F,v)$ is a q-defectless field.
\sn
b) \ If $vK$ is cofinal in $vF$ and $(K,v)$ is a c-defectless field,
then $d_c(F|K,v)=1$ and $(F,v)$ is a c-defectless field.
\end{theorem}

We will use $d(F|K)$, $d_c(F|K)$ and $d_q(F|K)$ for defect, completion
defect and quotient defect of $(F|K,v)$, respectively, when there is no
ambiguity for the valuation $v$.

%
%
\section{Valuation theoretic preliminaries}
For the basic facts of valuation theory, we refer the reader to [E],
[EP], [R], [W] and [Z-S].

%
%
\subsection{Henselization and completion}                \label{secthc}
%
Every finite extension $L$ of $(K,v)$ satisfies the \textbf{fundamental
inequality} (cf.\ [E]):
\begin{equation}                                   \label{fiq}
[L:K]\geq \sum_{i=1}^g {\rm e}_i {\rm f}_i,
\end{equation}
where $v_1,\ldots,v_g$ are the distinct extensions of $v$ from $K$ to
$L$, ${\rm e}_i=(v_iL:vK)$ are the respective ramification indices and
${\rm f}_i=[Lv_i:Kv]$ are the respective inertia degrees. If $g=1$ for
every finite extension $L|K$ then $(K,v)$ is called \textbf{henselian}.
This means that $(K,v)$ is henselian if and only if $v$ extends uniquely
to each algebraic extension of $K$. Therefore, every algebraically
closed valued field is trivially henselian.

Every valued field $(K,v)$ admits a \textbf{henselization}, that is, a
separable-algebraic extension field which is henselian and has the
universal property that it admits a unique embedding in every henselian
extension field of $(K,v)$. In particular, if $(L,w)$ is a henselian
extension field of $(K,v)$, then $(K,v)$ has a unique henselization in
$(L,w)$, which we will denote by $K^{h(w)}$. Further, all henselizations
of $(K,v)$ are isomorphic over $K$, so we often talk of \emph{the}
henselization of $(K,v)$ and just write $K^h$. We can also fix the
henselizations by consistently working inside a large algebraically
closed field. A valued field is henselian if and only if it is equal to
any (and thus all) of its henselizations.

An extension $(L|K,v)$ is {\bf immediate} if the canonical embeddings
$vK\hookrightarrow vL$ and $Kv\hookrightarrow Lv$ are onto, which we
also express by the less precise assertion that $vK=vL$ and $Kv=Lv$.
The henselization is an immediate extension.

If $L$ is a finite extension of $K$ and $v_1,\ldots,v_g$ are the
distinct extensions of $v$ from $K$ to $L$, then $(K,v)$ has a
henselization $K^{h(v_i)}$ in each henselization $L^{h(v_i)}$ of
$(L,v_i)$, and
\begin{equation}                               \label{localdegr}
[L:K]=\sum_{1\leq i\leq g} [L^{h(v_i)} : K^{h(v_i)}].
\end{equation}

We have:

\begin{lemma}                                 \label{Lh=L.Kh}
An algebraic extension of a henselian field is again henselian. If
$(L|K,v)$ is algebraic, then $(L.K^h,v)$ is the henselization of
$(L,v)$.
\end{lemma}

\parm
Let $(K,v)$ be any valued field. A valuation $w$ on $K$ is a
\textbf{coarsening} of $v$ if its valuation ring ${\mathcal{O}}_w$
contains the valuation ring ${\mathcal{O}}_v$ of $v$. If $H$ is a convex
subgroup of $vK$, then it gives rise to a coarsening $w$ with valuation
ring ${\mathcal{O}}_w:=\{x\in K \mid \exists\alpha\in H, \ \alpha \leq
vx\}$. Then $v$ induces a valuation $\ovl w$ on $Kw$ with valuation ring
${\mathcal{O}}_{\ovl{w}}:=\{xw \mid x\in{\mathcal{O}}_v \}$, and there
are canonical isomorphisms $wK\cong vK/H$ and $\ovl{w}(Kw)\cong H$. If
$(K,w)$ is any valued field and if $w'$ is any valuation on the residue
field $Kw$, then $w\circ w'$, called the \textbf{composition of $w$ and
$w'$}, will denote the valuation whose valuation ring is the subring of
the valuation ring of $w$ consisting of all elements whose $w$-residues
lie in the valuation ring of $w'$. In our above situation, $v$ is the
composition of $w$ and $\ovl{w}$.

The following fact is well known:
\begin{lemma}                         \label{comp-hen}
Take a valued field $(K,v)$ and a composition $v=w\circ\ovl{w}$. Then
$(K,v)$ is henselian if and only if $(K,w)$ and $(Kw,\ovl{w})$ are
henselian.
\end{lemma}

\parm
We conclude this section with a few results about the completion
$K^c$ of $K$. Like the henselization, also the completion is an
immediate extension.

\begin{lemma}                              \label{Lc=L.Kc}
If $(L|K,v)$ is finite, then the extension of $v$ from $K^c$ to $L.K^c$
is unique, and $(L.K^c,v)$ is the completion of $(L,v)$.
\end{lemma}
\begin{proof}
A finite extension of a complete valued field is again complete, hence
$(L.K^c,v)$ is complete, for any extension of $v$ from $K^c$ to $L.K^c$.
On the other hand, since $(L|K,v)$ is finite, the value group $vK$ is
cofinal in $vL$, which implies that the completion of $(L,v)$
must contain the completion of $(K,v)$. As it also contains $L$, it
contains $L.K^c$. Thus, $(L.K^c,v)$ is the completion of $(L,v)$, which
also implies that the extension of $v$ from $K^c$ to $L.K^c$ is unique.
\end{proof}

A proof of the following theorem can be found in [W] (Theorem 32.19):

\begin{lemma}                        \label{hchc=hc}
The completion of a henselian field is henselian too. Consequently,
\[
(K^{hc})^{hc} = K^{hc}.
\]
Moreover, a henselian field is separable-algebraically closed in its
completion.
\end{lemma}

%
%
\subsection{Defect and defectless fields}        \label{secDefect}
Assume that $(L|K,v)$ is a finite extension such that $v$ extends
uniquely from $K$ to $L$, then the Lemma of Ostrowski (cf.\ [EN], [R])
says that
\[
[L:K]=p^\nu(vL:vK)[Lv:Kv], \ \mbox{for some integer} \ \nu\geq 0
\]
where $p=\chara Lv$ if it is positive and $p=1$ otherwise. The factor
$d(L|K,v):=p^\nu$ is called the \textbf{defect} of the extension
$(L|K,v)$. If $d(L|K,v)=1$, then $L|K$ is called a \textbf{defectless
extension}. More generally (i.e., for $g\geq 1$), a finite extension
$(L|K,v)$ is called defectless if equality holds in (\ref{fiq}). A
valued field $(K,v)$ is said to be a \textbf{defectless},
\textbf{separably defectless} or \textbf{inseparably defectless field}
if every finite, finite separable or finite purely inseparable,
respectively, extension of $K$ satisfies equality in the fundamental
inequality (\ref{fiq}). One can trace this back to the case of unique
extensions of the valuation; for the proof of the following theorem, see
[K2] (a partial proof was already given in [E]):

\begin{lemma}                                        \label{dl-hdl}
A valued field is defectless if and only if its henselization is defectless.
The same holds for ``separably defectless" and ``inseparably defectless"
in place of ``defectless".
\end{lemma}
\n
Therefore, the Lemma of Ostrowski shows that:

\begin{corollary}                                   \label{corKv=0}
Every valued field $(K,v)$ with $\chara Kv=0$ is a defectless field.
\end{corollary}

The following lemma shows that the defect is multiplicative. This is a
consequence of the multiplicativity of the degree of field extensions
and of ramification index and inertia degree.

\begin{lemma}                               \label{dmult}
Let $K\subset L\subset M$ be fields and $v$ extends uniquely from $K$ to
$M$. Then
\[
d(M|K,v)=d(M|L,v)\cdot d(L|K,v)\,.
\]
In particular, $(M|K,v)$ is defectless if and only if $(M|L,v)$ and
$(L|K,v)$ are defectless.
\end{lemma}

Using this lemma together with Lemma~\ref{dl-hdl}, one easily shows:

\begin{lemma}
Every finite extension of a defectless field is again a defectless
field.
\end{lemma}

The following theorem is proved in [K5], where it is stated with the
additional hypothesis ``$\chara K=p>0$''. For our purpose in this paper,
we state it in general and include the proof.

\begin{lemma}                              \label{csdd}
Let $(K,v)$ be a henselian valued field. Then $K$ is separably defectless
if and only if $K^c$ is defectless.
\end{lemma}
\begin{proof}
In view of Corollary~\ref{corKv=0}, we may assume that $\chara Kv=p>0$.
Since $K$ is a henselian field, the same holds for $K^c$
(Lemma~\ref{hchc=hc}). The field $K^c$ is defectless if and only if it
is separably defectless, indeed, this is trivially true when $\chara
K=0$, and in the case of positive characteristic it is implied by
Theorem 5.1 of [K5]. Thus it suffices to prove that $K^c$ is a separably
defectless field if and only if $K$ is.

\parm
Let $L|K$ be an arbitrary finite separable extension. The henselian
field $K$ is separable-algebraically closed in $K^c$
(Lemma~\ref{hchc=hc}). Consequently, every finite separable extension of
$K$ is linearly disjoint from $K^c$ over $K$, whence
\[
[L.K^c:K^c]=[L:K].
\]
On the other hand, $L.K^c=L^c$ by Lemma~\ref{Lc=L.Kc}. Consequently,
\[
(v(L.K^c):vK^c)[(L.K^c)v:K^cv]=(vL:vK)[Lv:Kv].
\]
Assume that $K^c$ is a separably defectless field. Then $(L.K^c|K^c,v)$
is defectless, i.e., $[L.K^c:K^c]=(v(L.K^c):vK^c)[(L.K^c)v:K^cv]$.
Hence, $[L:K]=(vL:vK)[Lv:Kv]$, showing that $L|K$ is defectless. We have
shown that $K$ is separably defectless if $K^c$ is.

\parm
Now assume that $K^c$ is not a separably defectless field. Then there
exists a finite Galois extension $L'|K^c$ with nontrivial defect. In
view of Lemma~\ref{dmult}, we may assume that the extension is Galois
(after passing to the normal hull if necessary). We take an irreducible
polynomial $f=X^n+c_{n-1}X^{n-1}+\cdots+c_0\in K^c[X]$ of
which $L'$ is the splitting field. For every $\alpha\in vK$ there are
$d_{n-1},\ldots,d_0\in K$ such that $v(c_i-d_i)\geq \alpha$. If $\alpha$
is large enough, then by Theorem 32.20 of [W], the splitting fields of
$f$ and $g=X^n+d_{n-1}X^{n-1}+\cdots+d_0$ over the henselian field $K^c$
are the same. Consequently, if $L$ denotes the splitting field of $g$
over $K$, then $L'=L.K^c=L^c$. We obtain
\begin{eqnarray*}
[L:K]&\geq& [L.K^c:K^c]=[L':K^c] \\
&>&(vL':vK^c)[L'v:K^cv]=(vL^c:vK^c)[L^cv:K^cv]=(vL:vK)[Lv:Kv].
\end{eqnarray*}
That is, the separable extension $L|K$ is not defectless. Hence, $K$ is
not a separably defectless field.
\end{proof}

The following lemma describes the behaviour of the defect under
composition of valuations:

\begin{lemma}                         \label{comp-def}
Take a finite extension $(L|K,v)$ of henselian fields and a
coarsening $w$ of $v$ on $L$. Then
\[
d(L|K,v)\>=\>d(L|K,w)\cdot d(Lw|Kw,\ovl{w})\>.
\]
In particular, if $d(L|K,v) =1$, then $d(L|K,w) =1$ for every
coarsening $w$ of $v$.
\end{lemma}
\begin{proof}
Since $(L,v)$ and $(K,v)$ are henselian by assumption, also
$(L,w)$, $(K,w)$, $(Lw,\ovl{w})$ and $(Kw,\ovl{w})$ are henselian by
Lemma~\ref{comp-hen}. Therefore, we can compute:
\begin{eqnarray*}
d(L|K,v) & = & \frac{[L:K]}{(vL:vK)[Lv:Kv]}\>=\>
\frac{[L:K]}{(wL:wK)(\ovl{w}(Lw):\ovl{w}(Kw))[(Lw)\ovl{w}:(Kw)\ovl{w}]}\\
 & = & \frac{[L:K]}{(wL:wK)[Lw:Kw]}\cdot\frac{[Lw:Kw]}
{(\ovl{w}(Lw): \ovl{w}(Kw))[(Lw)\ovl{w}:(Kw)\ovl{w}]}\\[.2cm]
 & = & d(L|K,v)\cdot d(Lw|Kw,\ovl{w})\>.
\end{eqnarray*}
\end{proof}

In the next lemma, the relation between immediate and defectless extensions
is studied.
\begin{lemma}                       \label{lem-imm-def}
Take an arbitrary immediate extension $(F|K,v)$ of valued fields, and
$(L|K,v)$ a finite extension such that $[L:K]=(vL:vK)[Lv:Kv]$. Then
$F|K$ and $L|K$ are linearly disjoint, the extension of $v$ from $F$ to
$L.F$ is unique, $(L.F|F, v)$ is defectless, and $(L.F|L, v)$ is
immediate. Moreover,
\[
[L.F : F] = [L : K]\,,
\]
i.e., $F$ is linearly disjoint from $L$ over $K$.
\end{lemma}
\begin{proof}
$v(L.F)$ contains $vL$ and $(L.F)v$ contains $Lv$. On the other hand, we
have $vF=vK$ and $Fv=Kv$ by hypothesis. Therefore,
\begin{eqnarray*}
[L.F : F]&\geq&(v(L.F) : vF)\cdot[(L.F)v : Fv]  \\
&\geq&(vL : vK)\cdot[Lv : Kv] = [L : K]\geq [L.F : F]
\end{eqnarray*}
hence equality holds everywhere. This shows that $[L.F : F] = [L : K]$
and that $L.F|F$ is defectless with unique extension of the valuation.
Furthermore, it follows that $v(L.F)=vL$ and $(L.F)v =Lv$, i.e., $L.F|L$
is immediate.
\end{proof}

The reader should note that if the finite extension $L|K$ is not normal
and there are more than one extension of $v$ from $K$ to $L$, then
$d(L|K,v)=1$ does not imply that equality holds in (\ref{fiq}). It may
happen that for one extension of $v$ the henselian defect is $1$ while
for another extension it is $>1$. In this case, the henselian defect
depends on the chosen extension of $v$ from $K$ to $L$. On the other
hand, this will not happen when $L|K$ is normal.

Applying the lemma to purely inseparable extensions $L|K$, we obtain:

\begin{corollary}                           \label{idimmsep}
Every immediate extension of an inseparably defectless field is
separable.
\end{corollary}
\begin{proof}
If $(K,v)$ is an inseparably defectless field and $(F|K,v)$ an
immediate extension, then every finite purely inseparable extension
$(L|K,v)$ satisfies $[L:K]=(vL:vK)[Lv:Kv]$, and $L$ is therefore
linearly disjoint from $F$ over $K$ by the previous lemma. It follows
that also $K^{1/p^{\infty}}$ is linearly disjoint from $F$ over $K$.
\end{proof}

\parm
In the following we give two basic examples for extensions with defect
$>1$ (one can find more nasty examples in [K5] and [K7]).
The following is due to F.~K.~Schmidt.

\begin{example}                             \label{exampFKS}
{\rm We consider $\Fp((t))$ with its canonical valuation $v=v_t\,$.
Since $\Fp((t))|\Fp(t)$ has infinite transcendence degree, we can choose
some element $s\in\Fp((t))$ which is transcendental over $\Fp(t)$. Since
$(\Fp((t))|\Fp(t),v)$ is an immediate extension, the same holds for
$(\Fp(t,s)|\Fp(t),v)$ and thus also for $(\Fp(t,s)|\Fp(t,s^p),v)$. The
latter extension is purely inseparable of degree $p$ (since $s,t$ are
algebraically independent over $\Fp\,$, the extension $\Fp(s)|\Fp(s^p)$
is linearly disjoint from $\Fp(t,s^p)|\Fp(s^p)\,$). Hence, there is only
one extension of the valuation $v$ from $\Fp(t,s^p)$ to $\Fp(t,s)$. So
we have $e=f=g=1$ for this extension and consequently, its defect is
$p$}.
\end{example}

A defect can appear ``out of nothing'' when a finite extension
is lifted through another finite extension:

\begin{example}                             \label{exampindef}
{\rm In the foregoing example, we can choose $s$ such that $vs>1=vt$.
Now we consider the extensions $(\Fp(t,s^p)|\Fp(t^p,s^p),v)$ and
$(\Fp(t+s,s^p) |\Fp(t^p,s^p),v)$ of degree $p$. Both are defectless:
since $v\Fp(t^p,s^p) =p\Z$ and $v(t+s)=vt=1$, the index of $v\Fp(t^p,
s^p)$ in $v\Fp(t,s^p)$ and in $v\Fp(t+s,s^p)$ must be (at least) $p$.
But $\Fp(t,s^p).\Fp(t+s, s^p) = \Fp(t,s)$, which shows that the
defectless extension $(\Fp(t,s^p)| \Fp(t^p,s^p),v)$ does not remain
defectless if lifted up to $\Fp(t+s,s^p)$ (and vice versa)}.
\end{example}

%
%
\subsection{Defect of h-finite extensions}
%
For a finite extension $(L|K,v)$ such that $v$ extends uniquely from $K$
to $L$, the defect measures how far the fundamental inequality
(\ref{fiq}) is from being an equality. More generally, this can be done
for every algebraic extension $(L|K,v)$ such that $(L^h|K^h,v)$ is
finite, i.e., $(L|K,v)$ is an h-finite extension. This requires that we
work with a fixed extension of $v$ to the algebraic closure $\tilde{K}$,
which in turn determines the henselizations of $K$ and all its algebraic
extensions. In this case, we set
\[
d(L|K,v):=\frac{[L^h:K^h]}{(vL:vK)[Lv:Kv]}.
\]
Since the henselization is an immediate extension, we have that
$vK^h=vK$, $K^h v=Kv$, $vL^h=vL$ and $L^h v=Lv$. Therefore $d(L|K,v)$
is equal to the earlier defined defect of the extension $(L^h|K^h,v)$.
It is called the \textbf{henselian defect} or just the \textbf{defect}
of $(L|K,v)$ as $d(L|K,v)=d(L^h|K^h,v)$ when $v$ extends uniquely from
$K$ to $L$.

In general, the defect can increase or decrease if an h-finite extension
is lifted up through another extension. A defectless extension may turn
into an extension with nontrivial defect after lifting up through an
algebraic extension (as seen in Example~\ref{exampindef}). On the other
hand, every h-finite extension with nontrivial defect of a valued field
$(K,v)$ becomes trivial and thus defectless if lifted up to the
algebraic closure $\tilde{K}$. At least we can show that if the defect
decreases, then there is no further descent after a suitable finitely
generated extension.

As a preparation, we need the following fact which at first glance may
appear to be obvious. But a closer look reveals that proving it is more
difficult than expected. In order to get a feeling for the hidden
difficulties, the reader should note that if $K(x)|K$ is a simple
transcendental extension and we take an element $c$ in the henselization
of $K(x)$ which is algebraic over $K$, it may not lie in the
henselization of $K$ (cf.\ Theorem 1.3 of [K3]).

\begin{lemma}                               \label{fofdefh}
Take an arbitrary extension $(L|K,v)$ and elements $c_1,\ldots,c_m\in
L^h$. Then there exist elements $d_1,\ldots,d_n\in L$ such that
$c_1,\ldots,c_m\in K(d_1,\ldots,d_n)^h$.
\end{lemma}
\n
This lemma is proved in [K8]. Now we are ready to prove the following
result.

\begin{lemma}                                       \label{eeu}
Let $(L|K,v)$ and $(F|K,v)$ be subextensions of a valued field extension
$(\Omega|K,v)$ such that $F|K$ is finitely generated and $L.F|L$ is
h-finite. Then there exists a finitely generated subextension $L_0|K$ of
$L|K$ such that for every subfield $L_1$ of $L$ containing $L_0$, the
following holds:
\begin{enumerate}
\item $[(L.F)^h:L^h] = [(L_1.F)^h:L_1^h]$,  
\item $(v(L.F):vL)\leq (v(L_1.F):vL_1)$,    
\item $\left[(L.F)v:Lv\right]\leq \left[(L_1.F)v:L_{1}v\right]$,
\item $d(L.F|L,v)\geq d(L_1.F|L_1,v)$.      
\end{enumerate}
\end{lemma}
\begin{proof}
Since $[(L.F)^h:L^h]$ is finite, the fundamental inequality (\ref{fiq})
shows that $(v(L.F):vL)$ and $[(L.F)v:Lv]$ are finite too. Hence there
exist $a_1,\ldots,a_r\in L.F$ such that
\[
v(L.F) = vL + \Z va_1 + \ldots + \Z va_r\;,
\]
and there exist $b_1,\ldots,b_s\in L.F$ such that
\[
(L.F)v = Lv(b_1 v,\ldots,b_s v)\;.
\]
In order to write out $a_1,\ldots,a_r,b_1,\ldots,b_s$ as elements of the
compositum $L.F$, we need finitely many elements $a'_1,\ldots,a'_k,
b'_1,\ldots,b'_\ell\in L$. Whenever $L_1\subseteq L$ is an extension of
$K$ which contains these elements, then $a_1,\ldots,a_r,b_1,\ldots,b_s
\in L_1.F$ and it follows that
\begin{eqnarray}                                           \label{inc3+}
(v(L_1.F):vL_1) & \geq & (v(L.F):vL)\>,\\
{[(L_1.F)v:L_{1}v]} & \geq & [(L.F)v:Lv]\>,                \label{inc4+}
\end{eqnarray}
the left hand sides not necessarily being finite.

\pars
Now if $L_1$ satisfies assertion 1 of our lemma,
then the left hand sides of (\ref{inc3+}) and (\ref{inc4+})
have to be finite, and we will have that
\begin{eqnarray*}
d(L_1.F|L_1,v) & = & \frac{[(L_1.F)^h:L_1^h]}
{(v(L_1.F):vL_1)\cdot \left[(L_1.F)v:L_{1}v\right]}\\[3pt]
& \leq & \frac{[(L.F)^h:L^h]}
{(v(L.F):vL)\cdot \left[(L.F)v:Lv\right]}
\;=\; d(L.F|L,v)\;.
\end{eqnarray*}

Since $F|K$ is finitely generated by assumption, we can write $F=
K(z_1,\ldots,z_t)$. Then $z_1,\ldots,z_t$ are algebraic over $L^h$, and
we take $c_1,\ldots,c_m\in L^h$ to be all of the coefficients appearing
in their minimal polynomials. By Lemma~\ref{fofdefh} there exist
elements $d_1,\ldots,d_n\in L$ such that $c_1,\ldots,c_m\in K(d_1,
\ldots,d_n)^h$. Hence as soon as $d_1,\ldots,d_n\in L_1$, $z_1,\ldots,
z_t$ are algebraic over $L_1^h$ with $[L_1^h(z_1,\ldots,z_t): L_1^h]=
[L^h(z_1,\ldots,z_t):L^h]$. It then follows that $F|L_1$ is algebraic,
so that $(L_1.F)^h=L_1^h.F$ by Lemma~\ref{Lh=L.Kh} (where we take
$L=L_1.F$ and $K=L_1$). This yields the equalities
\begin{eqnarray*}
[(L_1.F)^h:L_1^h] & = & [L_1.F^h:L_1^h]\>=\>[L_1^h(z_1,\ldots,z_t):L_1^h]
\>=\> [L^h(z_1,\ldots,z_t):L^h]\\
 & = & [L^h.F:L^h]\>=\>[(L.F)^h:L^h]\>,
\end{eqnarray*}
so that assertion 1 is satisfied. Hence if we set $L_0 :=
K(a'_1,\ldots,a'_k, b'_1,\ldots,b'_\ell, d_1,\ldots,d_n)$, then all
assertions of our lemma will be satisfied by every subfield
$L_1\subseteq L$ that contains $L_0\,$.
\end{proof}

%
%
\subsection{Transcendence bases of valued function fields}
\label{secVFF}
For the easy proof of the following lemma, see [B], chapter VI,
\S10.3, Theorem~1.
\begin{lemma}                                      \label{K(T)}
Let $(L|K,v)$ be an extension of valued fields. Take elements $x_i,y_j
\in L$, $i\in I$, $j\in J$, such that the values $vx_i\,$, $i\in I$,
are rationally independent over $vK$, and the residues $y_jv$, $j\in
J$, are algebraically independent over $Kv$. Then the elements
$x_i,y_j$, $i\in I$, $j\in J$, are algebraically independent over $K$.

Moreover, if we write
\[f\>=\> \displaystyle\sum_{k}^{} c_{k}\,
\prod_{i\in I}^{} x_i^{\mu_{k,i}} \prod_{j\in J}^{} y_j^{\nu_{k,j}}\in
K[x_i,y_j\mid i\in I,j\in J]\]
in such a way that for every $k\ne\ell$
there is some $i$ s.t.\ $\mu_{k,i}\ne\mu_{\ell,i}$ or some $j$ s.t.\
$\nu_{k,j}\ne\nu_{\ell,j}\,$, then
\begin{equation}                            \label{value}
vf\>=\>\min_k\, v\,c_k \prod_{i\in I}^{}
x_i^{\mu_{k,i}}\prod_{j\in J}^{} y_j^{\nu_{k,j}}\>=\>
\min_k\, vc_k\,+\,\sum_{i\in I}^{} \mu_{k,i} v x_i\;.
\end{equation}
That is, the value of the polynomial $f$ is equal to the least of the
values of its monomials. In particular, this implies:
\begin{eqnarray*}
vK(x_i,y_j\mid i\in I,j\in J) & = & vK\oplus\bigoplus_{i\in I}
\Z vx_i\\
K(x_i,y_j\mid i\in I,j\in J)v & = & Kv\,(y_jv\mid j\in J)\;.
\end{eqnarray*}
Moreover, the valuation $v$ on $K(x_i,y_j\mid i\in I,j\in J)$ is
uniquely determined by its restriction to $K$, the values $vx_i$ and
the residues $y_jv$.

\parm
Conversely, if $(K,v)$ is any valued field and we assign to the
elements $vx_i$ any values in an ordered abelian group extension of $vK$
which are rationally independent, then (\ref{value}) defines a valuation
on $F$, and the residues $y_jv$, $j\in J$, are algebraically independent
over $Kv$.
\end{lemma}

As a consequence of the above lemma and the fundamental inequality
(\ref{fiq}), we have:
\begin{corollary}[Abhyankar Inequality]              \label{fingentb}
Let $(F|K,v)$ be an extension of valued fields of finite transcendence
degree. Then the Abhyankar Inequality (\ref{wtdgeq}) holds. If in
addition $F|K$ is a function field and if equality holds in
(\ref{wtdgeq}), then the extensions $vF| vK$ and $Fv|Kv$ are
finitely generated.
\end{corollary}

The following is Lemma 2.8 of [K4]:

\begin{lemma}                          \label{trdc}
Let $(F|K,v)$ be a valued function field without transcendence defect
and $v=w\circ\overline{w}$, then $(F|K,w)$ and $(Fw|Kw,\overline{w})$
are valued function fields without transcendence defect.
\end{lemma}

A transcendence basis $T$ of an extension $(L|K, v)$ is called
\textbf{valuation transcendence basis}, if for every choice of finitely
many distinct elements $t_1,\ldots,t_n\in T$, the value of every
polynomial $f$ in $K[t_1,\ldots,t_n]$ is equal to the value of a summand
of $f$ of minimal value, i.e.,
\begin{equation}                         \label{eqVTB}
v(\sum_{\underline\nu}c_{\underline\nu}t_1^{\nu_1}\cdots
t_n^{\nu_n})=\min_{\underline\nu}v(c_{\underline\nu}t_1^{\nu_1}\cdots
t_n^{\nu_n}).
\end{equation}
By Lemma~\ref{K(T)}, every standard valuation transcendence basis is a
valuation transcendence basis.

\begin{lemma}
Let $(L|K, v)$ be an extension of valued fields of finite transcendence
degree. Then the following assertions are equivalent:
\begin{enumerate}
\item $(L|K,v)$ is an extension without transcendence defect.
\item $(L|K,v)$ admits a standard valuation transcendence basis,
\item $(L|K,v)$ admits a valuation transcendence basis.
\end{enumerate}
\end{lemma}
\begin{proof}
1.$\Rightarrow$ 2. was shown in the Introduction.
\sn
2.$\Rightarrow$ 3. follows from our remark preceding the lemma.
\sn
3.$\Rightarrow$ 1.: Let $T=\{t_1,\ldots,t_n\}$ be a valuation
transcendence basis of $(L|K,v)$. Hence $n= \trdeg L|K$. We can assume
that the numbering is such that for some $r\geq0$, the values
$vt_1,\ldots, vt_r$ are rationally independent over $vK$ and the values
of every $r+1$ elements in $T$ are rationally dependent over $vK$. That
is, for every $j$ such that $0<j\leq s := n-r$, there are integers
$\nu_j > 0$ and $\nu_{ij}, \ 1\leq i\leq r$, and a constant $c_j\in K$
such that the element
\[
t'_j := c_jt_{r+j}^{\nu_j} \prod_{i=1}^{r}t_i^{\nu_{ij}}
\]
has value 0. Observe that $r\leq\rr vL/vK$ and that $r+s= \trdeg L|K$.
Now assume that $(L|K, v)$ has nontrivial transcendence defect. Then
\[
s = \trdeg L|K- r\geq \trdeg L|K - \rr vL|vK > \trdeg Lv|Kv.
\]
This yields that the residues $t'_1v,\ldots,t'_sv$ are not
$Kv$-algebraically independent. Hence, there is a nontrivial polynomial
$g(X_1,\ldots,X_s)\in{\mathcal{O}}_K[X_1,\ldots,X_s]$ such that
$g(t'_1,\ldots,t'_s)v=0$. Hence $vg(t'_1,\ldots,t'_s)> 0$. After
multiplying with sufficiently high powers of every element $t_i , 1\leq
i\leq r$, we obtain a polynomial $f$ in $t_1,\ldots, t_n$ which violates
(\ref{eqVTB}). But this contradicts our assumption that $T$ be a
valuation transcendence basis. Consequently, $(L|K, v)$ can not
have a nontrivial transcendence defect.
\end{proof}

%
%
\subsection{The defect of valued function fields}  \label{secDefFF}
Let $(F|K,v)$ be a subhenselian function field  without transcendence
defect. In the following we will show that the defect $d(F|K,v)$,
defined in (\ref{Tdef}), is finite and equal to the henselian defect
$d(F|K(T),v)$ for every standard valuation transcendence basis $T$
of $(F|K,v)$.

\begin{lemma}                               \label{lpr}
Take an extension $(K(T)|K,v)$, where $T=\{x_i,y_j \mid
i\in I, j\in J\}$ satisfies (\ref{valindep}). Let $v_1,\ldots,v_g$
be the extensions of $v$ from $K(T)$ to $L(T)$. Then $v_1,\ldots,v_g$
are uniquely determined by their restrictions to $L$, and these
restrictions are precisely the extensions of $v$ from $K$ to $L$.
Moreover, for $1\leq i\leq g$,
\begin{eqnarray}
d(L(T)|K(T),v_i) & = & d(L|K,v_i)\,,    \label{d_hK(T)}\\
e(L(T)|K(T),v_i) & = & e(L|K,v_i)\,,    \label{eK(T)}\\
f(L(T)|K(T),v_i) & = & f(L|K,v_i)\,,    \label{fK(T)}\\
v_i L(T) = v_iL + vK(T)\;\;&\mbox{and}&\;\;
L(T) v_i = Lv_i\,.\,K(T)v\;.           \label{v+v,o.o}
\end{eqnarray}
\end{lemma}
\begin{proof}
The first two assertions follow from Lemma~\ref{K(T)}, which also shows
that $vK(T)= vK\,\oplus\,\bigoplus_{i\in I}\Z vx_i$ and $v_iL(T)=
v_i L\,\oplus\,\bigoplus_{i\in I}\Z vx_i$. Hence, $v_iL(T)/
vK(T)\isom v_iL/vK$, which proves equation (\ref{eK(T)}).
Again by Lemma~\ref{K(T)}, $K(T)v=Kv(y_j v\mid j\in J)$ and
$L(T)v_i=Lv_i(y_j v\mid j\in J)$. Since the elements $y_j v$
are algebraically independent over $Kv$ and $Lv_i|Kv$ is algebraic,
$Kv(y_j v\mid j\in J)|Kv$ is linearly disjoint from $Lv_i|Kv$, which
yields (\ref{fK(T)}). Also, (\ref{v+v,o.o}) follows immediately
from the above described form of the value groups and residue fields.

Since the elements of $T$ are algebraically independent over $K$, the
extension $K(T)|K$ is linearly disjoint from $\tilde{K}|K$ and thus,
$[L(T): K(T)]=[L:K]$. In view of Lemma~\ref{Lh=L.Kh}, we have that

\begin{equation}                            \label{LThLh}
[L(T)^{h(v_i)}:K(T)^{h(v_i)}]=[L^{h(v_i)}.K(T)^{h(v_i)}: K^{h(v_i)}.
K(T)^{h(v_i)}]\leq [L^{h(v_i)}:K^{h(v_i)}]\;.
\end{equation}
But from (\ref{localdegr}) we obtain that
\[
\sum_{1\leq i\leq g}^{} [L(T)^{h(v_i)}:K(T)^{h(v_i)}] =
[L(T): K(T)]=[L:K]= \sum_{1\leq i\leq g}^{}
[L^{h(v_i)}:K^{h(v_i)}]
\]
which shows that equality must hold in (\ref{LThLh}). Now
(\ref{d_hK(T)}) follows from the definition of the henselian defect.
\end{proof}

To facilitate notation, we will from now on assume that all valued field
extensions of $(K,v)$ are contained in a large algebraically closed
valued field extension of $(K,v)$ and their henselizations are taken
within this extension. This enables us to suppress the mentioning of the
valuation.

\begin{lemma}                                   \label{d3}
Let $F|K$ be a subhenselian function field without transcendence defect.
Then for every standard valuation transcendence basis $T$ of $F|K$
there exists a finite extension $K_T$ of $K$ such that for every
algebraic extension $L$ of $K$ containing $K_T$, the following
holds:
\begin{enumerate}
\item the extension $L.F|L(T)$ is defectless
\item if $L|K$ is h-finite, then $d (L.F|K(T)) =
d(L(T)|K(T)) = d(L|K)$.
\end{enumerate}
\end{lemma}
\begin{proof}
Assume that $L|K$ is an h-finite extension such that
$L.F|L(T)$ is defectless. Then
\[d(L.F|K(T))=d(L.F|L(T))\cdot d(L(T)|K(T))
= d(L(T)|K(T)) = d(L|K)\;,\]
where the last equation holds by Lemma~\ref{lpr}.
Hence we may restrict our attention to the fulfillment of assertion 1.

%
%

\pars
The extension $\tilde{K}.F|\tilde{K}(T)$ is defectless by
Theorem~\ref{thmGST} and Lemma~\ref{dl-hdl}.
By Lemma~\ref{eeu} there exists a finitely generated subextension
$L_0|K(T)$ of $\tilde{K}(T)|K(T)$ such that
\[
d(L_1.F|L_1)\leq d(\tilde{K}.F|\tilde{K}(T))=1
\]
whenever $L_0\subset L_1\subset \tilde{K}(T)$. Let $K_T$ be a
finitely generated algebraic (and hence finite) extension of $K$ such
that $L_0\subset K_T(T)$. Then $d(L.F|L(T))\leq d(\tilde{K}.F|
\tilde{K}(T))=1$ for every algebraic extension $L$ of $K$ which
contains $K_T$.
\end{proof}

\sn
{\bf Proof of Theorem~\ref{def}:}
\n
%
Take any transcendence basis $T_0$ of $F|K$. Then also $K(T_0) |K$ is
without transcendence defect because it has the same transcendence
degree as $F|K$, $vF/vK(T_0)$ is a torsion group, and $Fv|K(T_0)v$ is
algebraic. Hence $K(T_0)$ admits a standard valuation transcendence
basis $T$ over $K$. We compute:
\[
d(F|K(T_0)) \leq d(F|K(T_0))\cdot d(K(T_0)|K(T)) = d(F|K(T))\>.
\]
This shows that
\[
d(F|K) = \sup_T\> d(F|K(T))\>,
\]
where $T$ runs over standard valuation transcendence bases only. Since
$F$ is a subhenselian function field, every $d(F|K(T))$ is a finite
number. It remains to show that for any two standard valuation
transcendence bases $T_1$ and $T_2$,
\[
d(F|K(T_1)) = d(F|K(T_2))\;.
\]
We choose finite extensions $K_{T_1}$ and $K_{T_2}$
according to Lemma~\ref{d3}. Putting $L_0 = K_{T_1}.\,
K_{T_2}$ we get by Lemma~\ref{d3}:
\[
d(L_0.F|K(T_1)) = d(L_0|K) = d(L_0.F|K(T_2))
\]
and from this we deduce
\[
d(F|K(T_1)) = \frac{d(L_0.F|K(T_1))}{d(L_0.F|F)}
= \frac{d(L_0.F|K(T_2))}{d(L_0.F|F)} = d(F|K(T_2))\;.
\]
This proves (\ref{defe}) and that the defect is independent of the
chosen standard valuation transcendence basis.

Furthermore, using Lemmas~\ref{lpr} and~\ref{d3}, for any finite
extension $L$ of $K$ containing $K_{T_1}$ we observe the following:
\begin{eqnarray*}
d(L(T_2)|K(T_2))
& = & d(L|K) \>=\> d(L.F|K(T_1))\>=\>d(L.F|F)\cdot d(F|K(T_1))\\
& = & d(L.F|F)\cdot d(F|K(T_2)) \>=\> d(L.F|K(T_2))
\end{eqnarray*}
showing that
\[
d(L.F|L(T_2))\>=\>d(L.F|K(T_2))/d(L(T_2)|K(T_2))= 1\;.
\]
Hence every algebraic extension $L$ of $K_{T_1}$ satisfies
assertion 1 and also the first part of assertion 2, because
\[
d(F|K) = d(F|K(T_1)) =
\frac{d(L.F|K(T_1))}{d(L.F|F)} = \frac{d(L|K)}{d(L.F|F)}
\]
where the last equation holds by Lemma~\ref{d3}.
The second part of assertion 2 follows from
\begin{eqnarray*}
d(N|K) & = & d(N(T_1)|K(T_1))\\
& \leq & d(N.F|K(T_1)) = d(N.F|F)\cdot d(F|K(T_1))
= d(N.F|F)\cdot d(F|K)
\end{eqnarray*}
and the fact that equality holds for the finite extension $K_{T_1}$ of $K$.
\QED

%
%
\section{Completion Defect and Defect Quotient}
This section contains our results on completion defect and defect
quotients for finite (and more generally, h-finite) extensions as well
as for valued algebraic function fields (and more generally,
subhenselian function fields).

%
%
\subsection{The case of h-finite extensions} \label{sechF}
The following observations are immediate from the definitions. We have
that
\begin{equation}  \label{dmt}
d(L|K) = d_c(L|K)\cdot d_q(L|K)\;.
\end{equation}
Thus for h-finite extensions of q-defectless fields, the completion
defect equals the ordinary defect. Every h-finite extension $L|K$
satisfies:
\[
d_c(L|K) = d_c(L^h|K^h)\;\;\;\mbox{\ \ and\ \ }\;\;\;
d_q(L|K) = d_q(L^h|K^h)\;.
\]
Hence, $K$ is a c-defectless or q-defectless field if and only if
its henselization $K^h$ is a c-defectless or q-defectless field,
respectively. If also $M|L$ is h-finite, then we have that
$[M^{hc}:K^{hc}]=[M^{hc}:L^{hc}]\cdot [L^{hc}:K^{hc}]$, and from the
multiplicativity of of ramification index and inertia degree we obtain
the following analogue of Lemma~\ref{dmult}:
\begin{equation}                            \label{multcq}
d_c(M|K) \>=\> d_c(M|L)\cdot d_c(L|K)\;\mbox{\ \ and\ \ }\;
d_q(M|K) \>=\> d_q(M|L)\cdot d_q(L|K)\;.
\end{equation}

From this multiplicativity, one derives:

\begin{lemma}                                         \label{fin}
Let $(L|K,v)$ be an h-finite extension. Then $(K,v)$ is a q-defectless
field if and only if $(L|K,v)$ is q-defectless and $(L,v)$ is a
q-defectless field. The same holds for ``c-defectless'' instead of
``q-defectless''.
\end{lemma}

In passing, we make the following observation, which we will not need
further in this paper.

\begin{lemma}
If $L|K$ is a c-defectless and immediate h-finite extension, then
$L^h$ is a purely inseparable extension of $K^h$ included in its
completion $K^{hc}$.
\end{lemma}
\begin{proof}
If $L|K$ is an immediate h-finite extension, then $(vL:vK)[Lv:Kv]=1$ and
therefore, $d_c(L|K)=[L^{hc}:K^{hc}]$. If in addition $L|K$ is
c-defectless, then $K^{hc}=L^{hc}=L^h.K^{hc}$, which implies that
$L^h\subseteq K^{hc}$. Since $K^h$ is separable-algebraically closed in
its completion $K^{hc}$ by Lemma~\ref{hchc=hc}, the extension $L^h|K^h$
must be purely inseparable.
\end{proof}

\pars
The completion defect $d_c(L|K)$ and the defect quotient $d_q(L|K)$ are
integers dividing $d(L|K)$ and hence are powers of $p$. To see this, we
use that $[L^{hc} : K^{hc}] = [L.K^{hc} : K^{hc}] \leq [L.K^h : K^h] =
[L^h : K^h]\,$. This gives:
\begin{equation}                               \label{dhc<=dh}
d_c(L|K)\;=\;\frac{[L^{hc}:K^{hc}]}{(vL:vK)\cdot [Lv:Kv]}
\;\leq\;\frac{[L^h:K^h]}{(vL:vK)\cdot [Lv:Kv]}\;=\;d(L|K)\;.
\end{equation}
Since on the other hand, $d_c(L|K)$ is the defect of the extension
$L^{hc}|K^{hc}$, it is a power of $p$ and consequently a divisor of
$d(L|K)$. This yields that also $d_q(L|K) = d(L|K)d_c(L|K)^{-1}$ is an
integer dividing $d(L|K)$ and a power of $p$.

\pars
In (\ref{dhc<=dh}), equality holds if and only if
\begin{equation}                                  \label{ldis}
[L^{hc} : K^{hc}] = [L^h : K^h]\;,
\end{equation}
which in view of $L^{hc}=L^h.K^{hc}$ means that $L^h$ is linearly
disjoint from $K^{hc}$ over $K^h$. Since the henselian field $K^h$ is
relatively separable-algebraically closed in its completion,
equation~(\ref{ldis}) holds whenever $L^h|K^h$ is separable-algebraic;
hence it holds for every h-finite separable extension $L|K$. This
proves:

\begin{lemma}                             \label{hfin}
Every h-finite separable extension is q-defectless. In general,
an h-finite extension $L|K$ is q-defectless if and only if
Equation~(\ref{ldis}) holds.
\end{lemma}

We deduce:
\sn
{\bf Proof of Theorem~\ref{ims}:}
\n
Let $K$ be a c-defectless field. By Lemma~\ref{hfin}, we know that
every h-finite separable extension of $K$ is q-defectless, i.e., its
completion defect equals the ordinary defect. Thus every finite
separable extension of $K$ is defectless and consequently, $K$ is a
separably defectless field.

For the converse, assume that $K$ is separably defectless. Then by
Lemma~\ref{dl-hdl}, also its henselization is separably defectless. Now
Lemma~\ref{csdd} shows that $K^{hc}$ is defectless. By virtue of the
definition of the completion defect, this proves $K$ to be c-defectless.
\QED

We will need the following theorem from [K6]:

\begin{theorem}                             \label{disMT2}
Take $z\in \tilde{K}\setminus K$ such that
\[
v(a-z)\;>\;\{v(a-c)\mid c\in K\}
\]
for some $a\in K^h$. Then $K^h$ and $K(z)$ are not linearly
disjoint over $K$, that is,
\[
[K^h(z):K^h] < [K(z):K]
\]
and in particular, $K(z)|K$ is not purely inseparable.
\end{theorem}

With this theorem, we are able to prove:

\begin{lemma}                               \label{hcchc}
For every finite purely inseparable extension $L|K$, $L^c$ is linearly
disjoint from $K^{hc}$ over $K^c$, and
\[
[L^{hc}:K^{hc}]\>=\>[L^{ch}:K^{ch}]\>=\>[L^c:K^c]\>.
\]
The first equation holds more generally whenever $L|K$ is h-finite.
\end{lemma}
\begin{proof}
Take a finite purely inseparable extension $L|K$ and
assume that $L^c$ is not linearly disjoint from $K^{hc}$ over $K^c$.
Then there exists an intermediate field $N$ between $L$ and $K$ and an
element $z\in L\setminus N$, $z^p \in N$, such that
\[
z \notin N.K^c\mbox{ but } z \in N .K^{hc}\;.
\]
Since $z\not\in N.K^c=N^c$, the set $\{z-c \mid c\in N\}$ is bounded
from above. Since $z\in N.K^{hc}=N^{hc}$, there exists an element $a\in
N^h$ such that $v(a-z)=v(z-a)>v(z-c)$ for all $c\in N$. This implies
$v(a-c)=\min\{v(a-z),v(z-c)\}=v(z-c)$, so that $v(a-z)>\{v(a-c)\mid c\in
N\}$. Now Theorem~\ref{disMT2}, with $N$ in place of $K$, proves that
$N^h$ and $N(z)$ are not linearly disjoint over $N$, which is a
contradiction since $N(z)|N$ is purely inseparable, while the
henselization of a valued field is a separable extension.

We have proved that $L^c$ is linearly disjoint from $K^{hc}$ over $K^c$.
Since $L^c \subseteq L^{hc}$ and $L.K^{hc}=L^{hc}$, we also have that
$L^c.K^{hc}= L^{hc}$. This together with the fact we have just proved
implies that $[L^{hc}:K^{hc}]=[L^c:K^c]$.

Now we observe that $K^c\subseteq K^{ch}\subseteq K^{ch}$ and $L^c.K^{ch}
=L^{ch}$, which yields that
\[
[L^{hc}:K^{hc}]\>\leq\>[L^{ch}:K^{ch}]\>\leq\>[L^c:K^c]\>=\>
[L^{hc}:K^{hc}]\>,
\]
so equality holds everywhere.

\parm
Now take an arbitrary h-finite extension $L|K$, and take $L_s|K$ to be
its maximal separable subextension. By what we have seen earlier, the
finite extension $L_s^h$ is linearly disjoint from $K^{hc}$ over $K^h$.
Since $L|K$ is h-finite, the subextension $L_s^h|K^h$ is finite, and
since $L_s^h.K^{hc}=L_s^{hc}$, we obtain that $[L_s^{hc}:K^{hc}]=
[L_s^h:K^h]$. Since $K^h\subseteq K^{ch}\subseteq K^{ch}$
and $L^h.K^{ch}=L^{ch}$, we find that
\[
[L_s^{hc}:K^{hc}]\>\leq\>[L_s^{ch}:K^{ch}]\>\leq\>[L_s^h:K^h]\>=\>
[L_s^{hc}:K^{hc}]\>,
\]
so equality holds everywhere.

We observe that $[L_s:L]$ must be finite since $L^h|K^h$ is finite by
assumption and the purely inseparable extension $L$ is linearly disjoint
from the separable extension $L_s^h$ over $L_s$. Thus, applying what we
have shown in the first part of the proof, with $L_s$ in place of $K$,
we obtain that $[L^{hc}:L_s^{hc}]=[L^{ch}:L_s^{ch}]$. Therefore,
$[L^{hc}:K^{hc}]=[L^{hc}:L_s^{hc}]\cdot [L_s^{hc}:K^{hc}]
=[L^{ch}:L_s^{ch}]\cdot [L_s^{ch}:K^{ch}]=[L^{ch}:K^{ch}]$.
\end{proof}

\sn
{\bf Proof of Proposition~\ref{ddcdch}:}
\n
Take an h-finite extension $L|K$. We have:
\begin{eqnarray*}
d(L^c|K^c) & = & \frac{[L^{ch}:K^{ch}]}{(vL^c:vK^c)\cdot [L^cv:K^cv]}
\>=\>\frac{[L^{ch}:K^{ch}]}{(vL:vK)\cdot [Lv:Kv]}\\
& = & \frac{[L^{hc}:K^{hc}]}{(vL:vK)\cdot [Lv:Kv]} \>=\>d_c(L|K)\>,
\end{eqnarray*}
where the second equality holds since the completion is an immediate
extension, and the third equality is taken from the previous lemma.

The second assertion of the proposition has been proven in
Lemma~\ref{hfin}.   \QED

\mn
{\bf Proof of Proposition~\ref{cens}:}
\n
Take a finite extension $L|K$, and $L_s|K$ its the maximal separable
subextension of $L|K$. Then $L|L_s$ is purely inseparable, and as we
have see in the proof of Lemma~\ref{hcchc}, it must be finite. By
Lemma~\ref{hfin}, $d_q(L_s|K)=1$ and
\[
d_q(L|K) \>=\> d_q(L|L_s)\cdot d_q(L_s|K) \>=\> d_q(L|L_s)
\>=\> \frac{[L^h:L_s^h]}{[L^{hc}:L_s^{hc}]}\;.
\]
Since $L$ is linearly disjoint from $L_s^h$ over $L_s\,$, we find that
\[
[L^h:L_s^h]\>=\>[L.L_s^h:L_s^h]\>=\>[L:L_s]\>=\>[L:K]_{\rm insep}\>.
\]
Further, $L^c=L.L_s^c$ is a purely inseparable extension of $L_s^c$ and
therefore linearly disjoint from $L_s^{ch}$ over $L_s^c$. Using in
addition the first equation from Lemma~\ref{hcchc} and the fact that
$L_s^c=L_s.K^c$ is a separable extension, we obtain that
\[
[L^{hc}:L_s^{hc}]\>=\>[L^{ch}:L_s^{ch}]\>=\>[L^c.L_s^{ch}:L_s^{ch}]\>=\>
[L^c:L_s^c]\>=\>[L^c:K^c]_{\rm insep}\>.
\]
This proves the proposition.
\QED

\sn
{\bf Proof of Theorem~\ref{tcs}:}
\n
$K$ is q-defectless if and only if every finite extension $L|K$ is
q-defectless. In view of the multiplicativity (\ref{multcq}) of the
defect quotient and the fact that every finite extension is contained in
a finite normal extension, it follows that $K$ is q-defectless if and
only if every finite normal extension $L|K$ is q-defectless. Again by
multiplicativity, and by the fact that a normal extension $L|K$ admits
an intermediate field $N$ such that $N|K$ is purely inseparable and
$L|N$ is separable and thus q-defectless (Lemma~\ref{hfin}), it follows
that $K$ is q-defectless if and only if every finite purely inseparable
extension $L|K$ is q-defectless. By Proposition~\ref{cens}, this is the
case if and only if $[L:K]=[L^c:K^c]=[L.K^c:K^c]$, i.e., $L$ is linearly
disjoint from $K^c$ over $K^c$, for every finite purely inseparable
extension $L|K$. This holds if and only if $K^c|K$ is separable.
\QED

We will now consider the behaviour of the defects under coarsenings of
the valuation.

\begin{lemma}                   \label{coal}
Take a finite extension $(L|K,v)$ and a decomposition $v =
w\circ\ovl{w}$ of $v$ with nontrivial $w$. Then
\begin{eqnarray}
&&d_q(L|K,v) = d_q(L|K,w)\label{coa1}\\
&&d_c(L|K,v) =
d_c(L|K,w)\cdot d(Lw|Kw,\ovl{w})
\label{coa2}\;.
\end{eqnarray}
\end{lemma}
\begin{proof}
To prove the equation for the defect quotient, we use
Proposition~\ref{cens} together with the fact that for every nontrivial
coarsening $w$ of the valuation $v$, the completion $K^{c(w)}$ of $K$
with respect to $w$ coincides with the completion $K^{c}$ with respect
to $v$. We obtain:
\[
d_q(L|K,v) \>=\> \frac{[L:K]}{[L^{c(v)}:K^{c(v)}]}\>=\>
\frac{[L:K]}{[L^{c(w)}:K^{c(w)}]} \>=\> d_q(L|K,w)\;.
\]
This proves equation~(\ref{coa1}). Using this result, we compute:
\[
d_c(L|K,v)\>=\>\frac{d(L|K,v)}{d_q(L|K,v)}\>=\>
\frac{d(L|K,w)\cdot d(Lw|Kw,\ovl{w})}{d_q(L|K,w)}\>=\>
d_c(L|K,w)\cdot d(Lw|Kw,\ovl{w})\>,
\]
proving equation~(\ref{coa2}).
\end{proof}

\pars
\begin{lemma}           \label{ncc}
Let $(L|K,v)$ be a finite extension of henselian fields and assume that
$v$ admits no coarsest nontrivial coarsening. Then there is a nontrivial
coarsening $w$
such that
\[
d_c (L|K,w) \>=\> 1 \>.
\]
If in addition the extension $L|K$ is separable, this means that
\[
d(L|K,w) \>=\> 1 \;\;\mbox{ \ and \ \ }\;\; d(L|K,v)\>=\>
d(Lw|Kw,\ovl{w})\>,
\]
where $\ovl{w}$ is the valuation induced by $v$ on $Lw$.
\end{lemma}
\begin{proof}
First, we note that if $L|K$ is a finite separable extension with
$d_c(L|K,w)=1$, then $d(L|K,w)=1$ by Proposition~\ref{ddcdch}; in view
of Lemma~\ref{comp-def}, this implies that $d(L|K,v)=d(Lw|Kw,\ovl{w})$.

Next, we prove the first assertion of our lemma in two special cases:
\sn
\underline{Case 1:} \ $L = K(a)$ is separable. Let $f(X)\in K[X]$ be the
minimal polynomial of $a$ over $K$ and let $c_i$, $0\leq i\leq n$ be the
coefficients of $f$. Then by our hypothesis on the rank of $v$ there
exists a nontrivial coarsening $w$ of $v$ such that $w$ is trivial on
$k(c_0,\ldots,c_n)$, where $k$ denotes the prime field of $K$. This
shows that $fw$ is a separable polynomial over $Kw$ of the same degree
as $f$; moreover it is irreducible since if it were reducible, then the
same would follow for $f$ by Hensel's Lemma (as $(K,w)$ is henselian by
our hypothesis on $(K,v)$ and Lemma~\ref{comp-hen}). Hence in this
case, $[L:K] = [Lw:Kw]$ and consequently $d(L|K,w) = 1$ and $d_c(L|K,w)
= 1$.

\sn
\underline{Case 2:}\ \ $L|K$ is a purely inseparable extension of
degree $p$. If $d_q(L|K,v) = p$, we are done because then $d_c(L|K,v) =
1$ and we take $w=v$. So we assume that
%
%
%
$d_q(L|K,v) = 1$. Then by Proposition~\ref{cens}, $[L:K]_{\rm insep} =
[L^c : K^c]_{\rm insep}$, showing that $a$ cannot be an element of
$K^c$. Consequently, there is an element $\alpha\in vK$ such that
$v(a-b)<\alpha$ for all $b\in K$. By our hypothesis on the coarsenings
of $v$ there exists a nontrivial coarsening $w$ of $v$ such that $w(a-b)
= 0$, hence $aw \not= bw$ for all $b\in K$ (this is satisfied if the
coarsening corresponds to any proper convex subgroup of $vK$ which
includes $\alpha$). This shows $aw\notin Kw$ and thus $[Lw:Kw] = p,$
which yields that $d(L|K,w) = 1$ and $d_c(L|K,w)=1$.

\pars
It remains to treat the case where $L|K$ is not simple. We then write
$L=K(a_1,\ldots,a_n)$ with $n>1$ and such that for every $i<n$, the
extension $K(a_1,\ldots,a_{i+1})|K(a_1,\ldots,a_i)$ is of degree $p$ if
it is inseparable. Now we proceed by induction on $n$.
Suppose that we have already found a nontrivial coarsening $w'$ of $w$
such that $d_c(K(a_1,\ldots,a_{n-1})|K,w')=1$. Applying what we have
proved above, with $K(a_1,\ldots,a_{n-1})$ in place of $K$ and $w'$ in
place of $v$, we find a nontrivial coarsening $w$ of $w'$ such that
$d_c(L|K(a_1,\ldots,a_{n-1}),w)=1$. From Lemma~\ref{coal} we know that
also $d_c(K(a_1,\ldots,a_{n-1})|K,w)=1$ since $w$ is a coarsening of
$w'$. By the multiplicativity of the completion defect, we obtain that
$d_c(L|K,w)=1$. This completes the proof of our Lemma.
\end{proof}

%
%
\subsection{The case of subhenselian function fields}{\label{secSHFF}}
Let us introduce some useful notions. Take an element $z$ in some valued
field extension of $(K,v)$. Then $z$ is called \textbf{value
transcendental over $(K,v)$} if $vz$ is not a torsion element modulo
$vK$, and it is called \textbf{residue transcendental over $(K,v)$} if
$vz=0$ and $zv$ is transcendental. If either is the case, we call $z$
\textbf{valuation transcendental over $(K,v)$}. For example, the
elements $x_i$ from (\ref{valindep}) are value transcendental, and the
elements $y_j$ are residue transcendental over $(K,v)$.

\begin{lemma}
Let $K(T)|K$ be an extension of valued fields with standard valuation
transcendence basis $T$. Then for every element $b\in K(T)
\setminus K$, there exist elements $c',c\in K$ such that $c'(b-c)$
is valuation transcendental.
\end{lemma}
\begin{proof}
Take $b=f/g\in K(T)$ with $f,g\in K[T]$. By Lemma~\ref{K(T)}, the value
of the polynomials $f,g$ is equal to the minimum of the values of the
monomials in $f$ resp.\ $g$, and these monomials are uniquely
determined; we will call them $f_0$ and $g_0$. If $f_0$ differs from
$g_0$ just by a constant factor $c\in K$, then we set $h=f-cg$ and
observe that the monomial $h_0$ of least value in $h$ will not anymore
lie in $Kg_0$. If already $f_0\notin Kg_0$, then we put $c=0$, $h=f$ and
$h_0=f_0$. Note that $h\not= 0$ and thus $h_0\not= 0$ since by
hypothesis, $f/g\notin K$. We have that
\[
b-c\>=\>\frac{f}{g}-c\>=\> \frac{h}{g}\mbox{\ \ with\ \ }v\frac{h}{g}
\>=\> v\frac{h_0^{}}{g_0^{}} \;,
\]
and we know that in the quotient $h_0/g_0$, at least one element of
$T$ appears with a nonzero (integer) exponent. If at least one of
these appearing elements from $T$ is value transcendental, then
$h_0/g_0$ and thus also $b-c$ is value transcendental over $K$. In this
case, we set $c'=1$.

In the remaining case, we write
\[
\frac{h_0}{g_0} = d\cdot y_1^{e_1}\cdot \ldots  \cdot y_s^{e_s}
\;,\;\;e_1,\ldots ,e_s\in\Z\;,
\]
where $d\in K$, and $y_1,\ldots ,y_s$ are different residue
transcendental elements from $T$. Since the residues
${y_1}v,\ldots,{y_s}v$ are algebraically independent over $Kv$, this
shows that $h_0/dg_0$ and thus also $h/dg$ are residue transcendental
over $K$. Putting $c'=d^{-1}$, we obtain that $c'(b -c)$ is residue
transcendental over $K$.
\end{proof}

The following proposition is a part of the assertion
of Theorem~\ref{313}.

\begin{proposition}                            \label{2.7}
Take a subhenselian function field $F|K$ without transcendence
defect. If $K$ is a q-defectless field or $vK$ is not cofinal in
$vF$, then $F$ is a q-defectless field.
\end{proposition}
\begin{proof}
Take a standard valuation transcendence basis $T$ of $F|K$. In view of
Lemma~\ref{fin} we only have to show that $K(T)$ is a q-defectless field
(hence we may assume $F=K(T)$). This will follow from Theorem~\ref{tcs}
if we can show that the completion of $F$ is a separable extension.

As the first case, let us assume that $K$ is a q-defectless field and
that $vK$ is cofinal in $vF$. Then the completion $F^c$ of $F$
contains the completion $K^c$ of $K$. By our hypothesis on $K$ and
Theorem~\ref{tcs}, $K^{1/p^{\infty}}$ is linearly disjoint from $K^c$
over $K$. We want to show now that $K^{1/p^{\infty}}$ is even linearly
disjoint from $F^c$ over $K$. This will follow if we prove that
$K^{1/p^{\infty}}.K^c$ is linearly disjoint from $F^c$ over $K^c$.

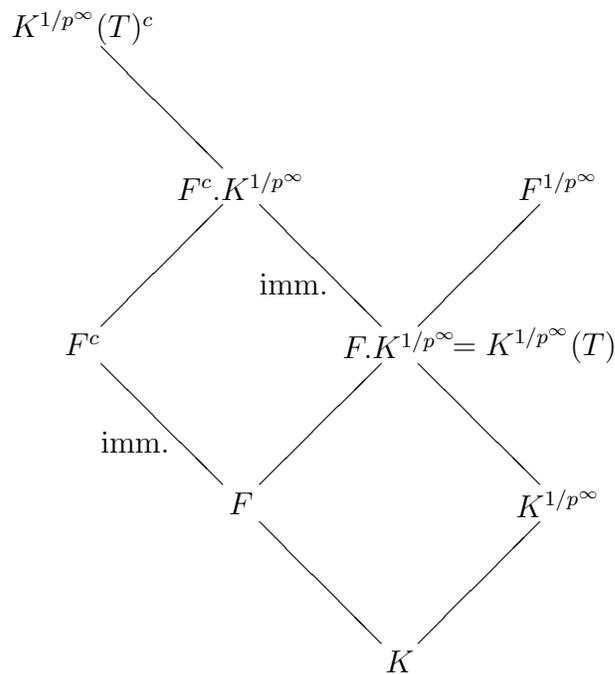
\begin{figure}[htb]
\begin{picture}(280,270)(-120,-80)
\put(60,0){\makebox(0,0){$F$}}
\put(0,60){\makebox(0,0){$F^c$}}
\put(120,-60){\makebox(0,0){$K$}}
\put(180,0){\makebox(0,0){$K^{1/p^{\infty}}$}}
\put(120,60){\makebox(0,0){$F.K^{1/p^{\infty}}$}}
\put(60,120){\makebox(0,0){$F^c.K^{1/p^{\infty}}$}}
\put(180,120){\makebox(0,0){$F^{1/p^{\infty}}$}}
\put(0,180){\makebox(0,0){$K^{1/p^{\infty}}(T)^c$}}
\put(7,53){\line(1,-1){46}}
\put(67,113){\line(1,-1){46}}
\put(7,173){\line(1,-1){46}}
\put(67,7){\line(1,1){46}}
\put(127,67){\line(1,1){46}}
\put(7,67){\line(1,1){46}}
\put(67,-7){\line(1,-1){46}}
\put(127,-53){\line(1,1){46}}
\put(127,53){\line(1,-1){46}}
\put(33,27){\makebox(0,0)[rt]{imm.}}
\put(93,87){\makebox(0,0)[rt]{imm.}}
\put(140,60){\makebox(0,0)[l]{$= K^{1/p^{\infty}}(T)$}}
\end{picture}
\end{figure}

Assume the contrary. Then there is a finite purely inseparable extension
$N$ of $K$ and an element $a\in K^{1/p^{\infty}}$ such that $a\in N.F^c
\setminus N.K^c$. Since $a \notin N.K^c = N^c$, the set $v(a-N)$ must be
bounded from above. Now $N.F^c = (N.F)^c$, hence there exists an element
$b\in N.F= N(T)$ such that $v(a-b) > v(a-N)$. But according to the
preceding Lemma, there exist elements $c,c'\in N$ such that $c'(b -c)$
is valuation transcendental. As $a$ is algebraic over $N$, this yields
that
\[
v(c'(a - c) - c'(b - c))\>=\> \min\{vc'(a - c),vc'(b -c)\}
\>\leq\> vc'(a - c)
\]
and consequently,
\begin{eqnarray*}
v(a-b) & = & v(c'a - c'b) -vc'\>=\> v(c'(a - c) - c'(b -c)) -vc'\\
& \leq & vc'(a - c) - vc' = v(a - c)\;,
\end{eqnarray*}
a contradiction as $v(a-b)>v(a-N)$. We have shown that
$K^{1/p^{\infty}}$ is linearly disjoint from $F^c$ over $K$.
Consequently, $F.K^{1/p^{\infty}}$ is linearly disjoint from $F^c$ over
$F$.

By Theorem~\ref{thmGST}, $F.K^{1/p^{\infty}} = K^{1/p^{\infty}}(T)$ is
an inseparably defectless field. On the other hand, the extension
$F^c.K^{1/p^{\infty}} |F.K^{1/p^{\infty}}$ is immediate since
$F^c.K^{1/p^{\infty}}$ is included in $K^{1/p^{\infty}}(T)^c$. Now
Corollary~\ref{idimmsep} shows that $F^{1/p^{\infty}}=
K(T)^{1/p^{\infty}}$ is linearly disjoint from $F^c.K^{1/p^{\infty}}$
over $F.K^{1/p^{\infty}}$. Putting this result together with what we
have already proved, we see that $F^{1/p^{\infty}}$ is linearly disjoint
from $F^c$ over $F$. Hence by Theorem~\ref{tcs}, $F$ is q-defectless.
This completes our proof in the first case.

\pars
In the remaining second case, $vK$ is not cofinal in $vF$, i.e.,
the convex hull of $vK$ in $vF$ is a proper convex subgroup of $vF$.
Consequently, there exists a nontrivial coarsening $w$ of the valuation
$v$ on $F$ which is trivial on $K$. Trivially, $(K,w)$ is a defectless
field, and so is $(F,w)$ according to Theorem~\ref{thmGST} since by
Lemma~\ref{trdc} it is a function field without transcendence defect
over $(K,w)$. Thus any finite purely inseparable extension is defectless
and thereby linearly disjoint from the $w$-completion $F^{c(w)}$ of $F$
since this is an immediate extension of $F$. On the other hand, the
topology induced by $v$ equals the topology induced by any
nontrivial coarsening of $v$, whence $F^{c(w)} = F^c$.
Consequently, $F^{1/p^{\infty}}$ is linearly disjoint from $F^c$.
By virtue of Theorem~\ref{tcs}, this completes our proof.
\end{proof}

On the basis of Proposition~\ref{2.7}, we are able to prove the
following lemma:
\begin{lemma}                                           \label{lpr'}
Let $K(T)|K$ be an extension of valued fields with standard valuation
transcendence basis $T$. Let $L$ be a finite extension of $K$. If
$vK$ is cofinal in $v(K(T))$, then
\begin{eqnarray*}
d_c(L(T)|K(T)) & = & d_c(L|K)\\
d_q(L(T)|K(T)) & = & d_q(L|K)\;.
\end{eqnarray*}
If $vK$ is not cofinal in $v(K(T))$, then
\begin{eqnarray*}
&d_c(L(T)|K(T)) = d(L(T)|K(T))=
d(L|K)&\\
&d_q(L(T)|K(T)) = 1\;.&
\end{eqnarray*}
\end{lemma}
\begin{proof}
If $vK$ is not cofinal in $v(K(T))$, the assertion follows from
Proposition~\ref{2.7} together with equations (\ref{d_hK(T)}) and
(\ref{dmt}). Let us assume now that $vK$ is cofinal in $v(K(T))$. Again
by equations (\ref{d_hK(T)}) and (\ref{dmt}), it suffices to prove the
first equality.
\begin{figure}[htb]
\begin{picture}(260,160)(-130,-10)
\put(9,0){\makebox(0,0){$K$}}
\put(9,40){\makebox(0,0){$K^{hc}$}}
\put(9,80){\makebox(0,0){$K^{hc}(T)$}}
\put(9,120){\makebox(0,0){$K(T)^{hc}$}}
\put(90,27){\makebox(0,0){$L$}}
\put(90,67){\makebox(0,0){$L^{hc}$}}
\put(90,107){\makebox(0,0){$L^{hc}(T)$}}
\put(90,147){\makebox(0,0){$L(T)^{hc}$}}
\put(24,5){\line(3,1){51}}
\put(24,45){\line(3,1){51}}
\put(30,87){\line(3,1){39}}
\put(30,127){\line(3,1){39}}
\put(9,7){\line(0,1){25}}
\put(9,47){\line(0,1){25}}
\put(9,87){\line(0,1){25}}
\put(90,34){\line(0,1){25}}
\put(90,74){\line(0,1){25}}
\put(90,114){\line(0,1){25}}
\put(-44,120){\makebox(0,0){$(K^{hc}(T))^{hc} =$}}
\put(110,147){\makebox(0,0)[l]{$= L.K(T)^{hc} =
(L^{hc}(T))^{hc}$}}
\put(110,107){\makebox(0,0)[l]{$= L.K^{hc}(T)$}}
\put(110,67){\makebox(0,0)[l]{$= L.K^{hc}$}}
\end{picture}
\end{figure}
\n
Using Lemma~\ref{lpr}and that $L^{hc}(T) = (L.K^{hc})(T) =
L.(K^{hc}(T))$, we obtain that
\begin{equation}                         \label{first}
d_c(L|K)\>=\> d(L^{hc}|K^{hc})\>=\>d(L^{hc}(T)|K^{hc}(T))\;.
\end{equation}
The complete field $K^{hc}$ is q-defectless by Theorem~\ref{tcs}. Hence
by Proposition~\ref{2.7}, $K^{hc}(T)$ is a q-defectless field too.
Consequently,
\begin{eqnarray*}
d(L^{hc}(T)|K^{hc}(T)) &=& d_c(L^{hc}(T)|K^{hc}(T))\>=\>
d((L^{hc}(T))^{hc}|(K^{hc}(T))^{hc})\\
&=& d(L(T)^{hc}|K(T)^{hc})\;,
\end{eqnarray*}
where the last equation holds since $(L^{hc}(T))^{hc} = L(T)^{hc}$ and
$(K^{hc}(T))^{hc} = K(T)^{hc}$. Putting this together with equation
(\ref{first}), we obtain that
\[
d_c(L|K)\>=\>d(L(T)^{hc}|K(T)^{hc})\>=\> d_c(L(T)|K(T))\;.
\]
\end{proof}

\sn
{\bf Proof of Theorem~\ref{def'}:}
\n
Take any transcendence basis $T_0$ of $F|K$. As in the proof of
Theorem~\ref{def'} it follows that $K(T_0) |K$ is without transcendence
defect and admits a standard valuation transcendence basis $T$ over $K$.
We compute:
\begin{eqnarray*}
d_c(F|K(T_0)) & \leq & d_c(F|K(T_0))\cdot d_c(K(T_0)|K(T)) \>=\>
d_c(F|K(T))\>,\\
d_q(F|K(T_0)) & \leq & \leq d_q(F|K(T_0))\cdot d_q(K(T_0)|K(T)) \>=\>
d_q(F|K(T))\>,
\end{eqnarray*}
showing that
\[
d_c(F|K) \>=\> \sup_T\> d_c(F|K(T))\;\;\mbox{\ \ and\ \ }\;\;
d_q(F|K) \>=\> \sup_T\> d_q(F|K(T))\>,
\]
where the supremum is only taken over all standard valuation
transcendence bases of $F|K$. For the proof of equations~(\ref{m}) it
suffices now to show that $d_c(F|K(T))$ is equal for all standard
valuation transcendence basis $T$, and that the same holds for the
defect quotient.

\pars
If $vK$ is not cofinal in $vF$, then by virtue of Proposition~\ref{2.7},
$K(T)$ is q-defectless. Using Theorem~\ref{def}, we obtain:
\[
d_q(F|K(T))\>=\>1 \;\;\mbox{\ \ and\ \ }\;\; d_c(F|K(T)) \>=\> d(F|K(T))
\>=\> d(F|K)\>,
\]
independently of the standard valuation transcendence basis $T$.

\pars
If $vK$ is cofinal in $vF$, then $K(T)^{hc}$ contains $K^{hc}$.
By Proposition~\ref{2.7}, $K^{hc}(T)$ is a q-defectless field.
From this we deduce, using Theorem~\ref{def} again:
\begin{eqnarray*}
d_c(F|K(T)) & = & d(F^{hc}|K(T)^{hc})\>=\>
d((F.K^{hc})^{hc}|(K^{hc}(T))^{hc})\\
& = & d_c(F.K^{hc}|K^{hc}(T))\>=\>d(F.K^{hc}|K^{hc}(T))\>=\>
d(F.K^{hc}|K^{hc})\;.
\end{eqnarray*}
For the defect quotient, this implies that
\[
d_q(F|K(T))\>=\>\frac{d(F|K(T))}{d_c(F|K(T))}\>=\>
\frac{d(F|K)}{d(F.K^{hc}|K^{hc})}\;.
\]
This completes the proof of equations (\ref{m}).

To prove equation~(\ref{mu}), we take an arbitrary standard valuation
transcendence basis $T$ and compute, using Theorem~\ref{def} together
with what we have just proved,
\[
d(F|K) \>=\> d(F|K(T)) \>=\> d_c(F|K(T))\cdot d_q(F|K(T))
\>=\> d_c(F|K)\cdot d_q(F|K)\;.
\]

\pars
For the remainder of the proof, we will assume that $vK$ is cofinal in
$vF$. Take $K'$ as in the assertion of Theorem~\ref{def}, and a finite
extension $L$ of $K$ containing $K'$. Choosing any standard valuation
transcendence basis $T$ of $F|K$, we know by Theorem~\ref{def} that
$d(L.F|L(T))=1$, whence
\[
d_c(L.F|L(T)) = 1\;.
\]
Using this and equation (\ref{m}) as well as the multiplicativity of
the completion defect (see (\ref{multcq}), we deduce:
\begin{eqnarray*}
d_c(L.F|F)\cdot d_c(F|K) & = &
d_c(L.F|F)\cdot d_c(F|K(T))\>=\> d_c(L.F|K(T))\\
& = & d_c(L.F|L(T))\cdot d_c(L(T)|K(T))\>=\> d_c(L(T)|K(T))\\
& = & d_c(L(T)|K(T))\>,
\end{eqnarray*}
where the last question holds by Lemma~\ref{lpr'}. This proves the first
equation of (\ref{m1}). The proof of the first equation of (\ref{m3}) is
obtained by just replacing the completion defect by the defect quotient
in the above argument.

\pars
The second equations in (\ref{m1}) and (\ref{m3}) are shown as it was
done for the defect in the proof of Theorem~\ref{def}.
\QED

\parm
As immediate consequences we get the following corollaries.
\begin{corollary}                       \label{310a}
Assume $F$ to be a subhenselian function field without transcendence
defect over a q-defectless field $K$. Then $d_q(F|K)$ is
trivial.
\end{corollary}
\begin{proof}
\end{proof}
\begin{corollary}
Every subhenselian function field $F|K$ without transcendence defect
satisfies
\begin{eqnarray*}
d_c(F|K) &=& d_c(F^h|K) \>=\> d_c(F^h|K^h)\>,\\
d_q(F|K) &=& d_q(F^h|K) \>=\> d_q(F^h|K^h)\>.
\end{eqnarray*}
\end{corollary}
\begin{proof}
Any standard valuation transcendence basis $T$ of $F|K$ is also a
standard valuation transcendence basis of $F^h|K$ and of $F^h|K^h$.
Hence, using Theorem~\ref{def'},
\begin{eqnarray*}
d_c(F|K) & = & d_c(F|K(T)) \>=\> d_c(F^h|K(T)^h)
\>=\> d_c(F^h|(K^h(T))^h)\\
& = & d_c((F^h)^h | (K^h(T))^h) \>=> d_c(F^h | K^h(T)) \>=>
d_c(F^h|K^h)
\end{eqnarray*}
and
\[
d_c(F^h|K(T)^h)\>=>d_c((F^h)^h|K(T)^h)\>=>d_c(F^h|K(T))\>=>d_c(F^h|K)\;.
\]
The assertions for the defect quotient are shown by combining the above
equations with the corresponding equations for the defect, and using
Theorem~\ref{def'}.
\end{proof}

\begin{corollary}                   \label{d2}
Let $E$ and $F$ be subhenselian function fields over $K$. If $E|F$ is
algebraic and $F|K$ has no transcendence defect, then $E|F$ is
h-finite and the following multiplicativity holds for the completion
defect and defect quotient:
\[
d_c(E|K) \>=\> d_c(E|F)\cdot d_c(F|K)\;\;\mbox{\ \ and\ \ }\;\;
d_q(E|K) \>=\> d_q(E|F)\cdot d_q(F|K)\;.
\]
\end{corollary}
\begin{proof}
First, we prove that $E|F$ is h-finite. As $E$ and $F$ are subhenselian
function fields, $E^h$ and $F^h$ are the henselizations of valued
function fields $E_0$ and $F_0$ over $K$. Since $E|F$ is algebraic, also
$E^h|F^h$ and $E_0.F_0|F_0$ are algebraic. As $E_0.F_0$ is also a
function field over $K$, $E_0.F_0|F_0$ is finite and the same holds for
$(E_0.F_0)^h|F_0^h$. But $F_0^h=F^h$, and since $E^h=E_0^h$ contains $F$
and thus also $F_0$, we see that $(E_0.F_0)^h=E^h$. We have proved that
$E^h|F^h$ is finite.

Taking any standard valuation transcendence basis $T$ of $F|K$ (which is
also a standard valuation transcendence basis of $E|K$ since $E|F$ is
algebraic), we compute
\[
d_c(E|K)\>=\>d_c(E|K(T))\>=\>d_c(E|F)\cdot d_c(F|K(T))\>=\>
d_c(E|F)\cdot d_c(F|K)\>,
\]
using Theorem~\ref{def'} and the multiplicativity of the
completion defect (see (\ref{multcq}). The proof for the defect
quotient is similar.
\end{proof}

\begin{corollary}                   \label{d1}
Let $F|K$ be a subhenselian function field without transcendence defect.
If $vK$ is cofinal in $vF$ then there exists a finite
extension $K'$ of $K$ such that
\[
d_c(K'.F|K)\>=\>d_c(K'|K)\;\;\mbox{\ \ and\ \ }\;\;
d_q(K'.F|K)\>=\>d_q(K'|K)\;.
\]
\end{corollary}
\begin{proof}
We take $K'$ as in the assertion of Theorem~\ref{def'} and apply
Corollary~\ref{d2} to the first equations in (\ref{m1}) and (\ref{m3}),
where we set $L = K'$.
\end{proof}

\pars
Now we are ready for the
\sn
{\bf Proof of Theorem~\ref{313}:}
\n
a): \
The assertion that $F$ is a q-defectless field has already been proven
in Proposition \ref{2.7}. Now we take a standard valuation transcendence
basis $T$ of $F|K$. Again by Proposition~\ref{2.7}, $K(T)$ is a
q-defectless field, hence in view of Theorem~\ref{def'},
\[
d_q(F|K) \>=\> d_q(F|K(T)) \>=\> 1\;.
\]

\sn
b) \ We assume that $vK$ is cofinal in $vF$ and that $K$ is a
c-defectless field. We choose $K'$ according to Corollary~\ref{d1}. Then
we have:
\[
d_c(K'.F|K) \>=\> d_c(K'|K) \>=\> 1\;.
\]
From Corollary~\ref{d2}, where we put $E= K'.F$, we get that
$d_c(F|K)=1$. On the other hand, if $F'$ is an arbitrary finite
extension of $F$, then
it is also a subhenselian function field without transcendence defect
over $K$ and consequently, like $F$ it satisfies $d_c(F'|K) = 1$. By
Corollary~\ref{d2}, we conclude that
\[d_c(F'|F) = 1\;.\]
This shows that $F$ is a c-defectless field.
\QED

The following theorem is a corollary to the Theorem~\ref{ims} and
Theorem~\ref{313}:

\begin{theorem}                                        
Let $F|K$ be a subhenselian function field without transcendence
defect. If $K$ is separably defectless and $vK$ is cofinal in $vF$, then
$F$ is separably defectless.
\end{theorem}

\bn
\bn
{\bf References}
\newenvironment{reference}%
{\begin{list}{}{\setlength{\labelwidth}{5em}\setlength{\labelsep}{0em}%
\setlength{\leftmargin}{5em}\setlength{\itemsep}{-1pt}%
\setlength{\baselineskip}{3pt}}}%
{\end{list}}
\newcommand{\lit}[1]{\item[{#1}\hfill]}
\begin{reference}

\lit{[B]} {Bourbaki, N.$\,$: {\it Commutative algebra}, Paris (1972).}

\lit{[BGR]} {Bosch, S. -- G\"untzer, U. - Remmert, R. : {\it
Non-Archimedean Analysis}, Berlin - Heidelberg - New York - Tokio
(1984).}

\lit{[E]} {Endler, O.$\,$: {\it Valuation theory}, Springer, Berlin (1972).}

\lit{[EP]} {Engler, A. -- Prestel, A.: {\it Valued fields},
Springer-Verlag, Inc., New-York (2005).}

\lit{[GMP]} {Green, B. -- Matignon, M. -- Pop, F. : {\it On valued
function fields I}, manuscripta math. \textbf{65} (1989), 357-376}

\lit{[G-R]} {Grauert, H. -- Remmert, R.: {\it \"Uber die Methode der
diskret bewreteten Ringe in der nicht-archimedischen Analysis.}, Invent.
Math. {\bf 2} (1942), 87--133}

\lit{[GRU]} {Gruson, L. : {\it Fibr\'es vectoriels sur un polydisque
ultram\'etrique}, Ann. Sci. Ec. Super., IV. Ser., \textbf{177} (1968),
45-89}

%

\lit{[K1]} {Kuhlmann, F.-V.$\,$: {\it Henselian function fields and tame
fields}, preprint (extended version of Ph.D.\ thesis), Heidelberg
(1990).}

\lit{[K2]} {Kuhlmann, F.-V.$\,$: {\it Valuation theory}. Book in
preparation. Preliminary versions of several chapters are available on
the web site\n
{\tt http://math.usask.ca/$\,\tilde{ }\,$fvk/Fvkbook.htm}}

\lit{[K3]} {Kuhlmann, F.-V.$\,$: {\it Value groups, residue fields and
bad places of rational function fields},
Trans.\ Amer.\ Math.\ Soc.\ 356 (2004), 4559-4600}

\lit{[K4]} {Kuhlmann, F.-V.$\,$: {\it Elimination of ramification I:
The Generalized Stability Theorem}, Trans. Amer. Math. Soc. {\bf 362}
(2010), 5697--5727.}

\lit{[K5]} {Kuhlmann, F.-V.$\,$: {\it A classification of Artin-Schreier
defect extensions and characterizations of
defectless fields}, Illinois J. Math. {\bf 54} (2010), 397--448.}

\lit{[K6]} {Kuhlmann, F.-V.$\,$: {\it Approximation of elements in
henselizations}, manuscripta math. {\bf 136} (2011), 461--474.}

\lit{[K7]} {Kuhlmann, F.-V.$\,$: {\it The Defect}, in: Commutative
Algebra -- Noetherian and non-Noetherian perspectives. Marco Fontana,
Salah-Eddine Kabbaj, Bruce Olberding and Irena Swanson (eds.), Springer
2011}

\lit{[K8]} {Kuhlmann, F.-V.$\,$: {\it Henselian elements}, in
preparation}

\lit{[Kh]} {Khanduja, S. K.: {\it An independence theorem in simple
transcendental extensions of valued fields.}, J. Indian Math. Soc. {\bf
63} (1997) 243--248.}

\lit{[M]} {Matignon, M. : {\it Genre et genre r$\acute{e}$siduel des
corps de fonctions valu$\acute{e}$s}, manu- scripta math. {\bf 58}
(1987), 179-214.}

\lit{[O]} {Ohm, J. : {\it The Henselian defect for valued function
fields}, Proc.\ Amer.\ Math.\ Soc., {\bf 107}(1989), no. 2, 299--308.}

\lit{[R]} {Ribenboim, P.$\,$: {\it Th\'eorie des valuations}, Les
Presses de l'Uni\-versit\'e de Montr\'eal, Montr\'eal, 1st ed.\ (1964),
2nd ed.\ (1968).}

\lit{[W]} {Warner, S.$\,$: {\it Topological fields}, Mathematics studies
\textbf{157}, North Holland, Amsterdam (1989).}

\lit{[Z-S]} {Zariski, O. - Samuel, P.$\,$: {\it Commutative Algebra},
Vol.~II, The University Series in Higher Mathematics, D. Von Nostrand
Co., Inc., Princeton, NJ-Toronto-London-New York.}

\end{reference}

\fvkadresse
\anadresse

\end{document}